\documentclass[12pt]{article}

\usepackage[dvipsnames]{xcolor}
\usepackage{notes}
\let\l\lambda

\def\O{\mathcal O}

\makeop{depth}
\makeop{tr}
\makeop{frk}

\let\cref\Cref
\let\mult e

\def\B{\mathcal B}
\def\l{\ell}

\let\obigwedge\bigwedge
\def\bigwedge{\mathop{\obigwedge}\nolimits}

\title{Antiampleness and ampleness of the Frobenius cokernel}
\author{Devlin Mallory}

\begin{document}
\maketitle
\abstract{We show that if $X$ is a smooth Fano variety containing a line or a conic with respect to $-K_X$, then the Frobenius cokernel $\B_X:=\coker(\O_X\to F_* \O_X)$ is not antiample; using this criteria, we show that the only smooth Fano threefolds with antiample Frobenius cokernel are $\P^3$ and the quadric threefold (in characteristic $p\neq 2$), thus answering a question raised in \cite{CRP}. 
We also show that for any smooth complete intersection $X\subset \P^n$ of degree $d_1,\dots,d_c$ such that $\sum d_i = n$ or $n-1$, the Frobenius cokernel is not antiample. 
We also study the kernels of the higher Cartier operators, and show that for $\P^n$ and quadric hypersurfaces, all the kernels of the higher Cartier operators are antiample, and thus that the full set of kernels of the Cartier operators cannot characterize projective space.
 Finally, we show that the Frobenius cokernel is ample if and only if the cotangent bundle is ample.}

\section{Introduction}

Let $X$ be a smooth projective variety over an algebraically closed field of characteristic~$p$, and write $F:X\to X$ for the (absolute) Frobenius.
A recurring question in positive-characteristic geometry is how the $\O_X$-module $F_* \O_X$ reflects the geometry of $X$; for example, if $X$ is a toric variety then $F_*\O_X$ is the direct sum of line bundles, and if $F_* L$ is a direct sum of line bundles for every line bundle $L$, then $X$ is a toric variety \cite{Achinger}.

In this tradition, \cite{CRP} examined the positivity properties of the dual of the Frobenius cokernel $\B_X := \coker(\O_X \to F_* \O_X)$, and raised the question of which varieties have ample $\B_X^\vee$.
They note that this cannot characterize projective space, since quadrics of dimension $\geq 3$ in odd characteristic also have ample $\B_X^\vee$, but ask (for example) which Fano threefolds, or which hypersurfaces, have ample $\B_X^\vee$. They also raise the question of whether the ampleness of the kernels of the higher Cartier operators (of which $\B_X^\vee$ is the first nontrivial kernel) can characterize projective space.
More recent work of \cite{CROZ} has focused on the case where $X$ is a toric variety, and shown that the only smooth toric varieties with ample (or even big) $\B_X^\vee$ are projective space, as well as defining and calculating some numerical measures of the positivity of~$\B_X^\vee$.

In this note, we find a geometric obstruction to the ampleness of $\B_X^\vee$:

\begin{thm}[\Cref{main}]
Let $X$ be a smooth projective variety over a field of characteristic $p>0$. If $X$ has a smooth positive-dimensional subvariety $Z$ with normal bundle $N_{Z/X}$ such that $(\det N_{Z/X})^{1-p}$ is effective, then $\B_X^\vee$ is not ample.
\end{thm}

This applies, in particular, whenever $X$ contains a smooth rational curve $C$ with $-K_X.C \leq 2$ (see \Cref{curveobstruction}).
This allows us settle the question of which Fano threefolds have ample $\B_X^\vee$, and likewise give an answer for Fano complete intersections of small index: 

\begin{thm}[\Cref{Fano3} and \Cref{CIs}]
Let $X$ be a smooth projective variety over an algebraically closed field of characteristic $p>0$.
If $X$ is
\begin{enumerate}
\item a Fano threefold that is not $\P^3$ or a quadric threefold, or
\item a complete intersection in $\P^n$ of dimension $\geq 2$ and multidegree $d_1,\dots,d_c\geq 2$, and $\sum d_i = n$ or $n-1$,
\end{enumerate}
then $\B_X^\vee$ is not ample.
\end{thm}

We also show that the ampleness of the kernels of the higher Cartier operators cannot characterize projective space, by showing that all the kernels of the higher Cartier operators are ample for quadric hypersurfaces:

\begin{thm}[\Cref{wedgeTX} and \Cref{TXhyper}]
Let $X$ be a smooth projective variety over an algebraically closed field of characteristic $p>0$. 
\begin{enumerate}
\item 
If $\bigwedge^i T_X$ is ample for some $i$, then $(\B^i_X)^\vee$ is ample.
\item
If $X$ is a smooth degree-$d$ hypersurface in $\P^n$, then $\bigwedge^i T_X$ is ample for $d\leq i \leq n-1$.
In particular, if $Q$ is a smooth quadric hypersurface in $\P^n$ then $\bigwedge^i T_Q$ is ample for $2\leq i \leq n-1$,
and if $p>2$ then this holds for all $i\geq 1$
\end{enumerate}
\end{thm}

Finally, we consider the opposite extreme, and characterize the \emph{ampleness} of $\B_X$:

\begin{thm}[\Cref{ample}]
Let $X$ be a smooth projective variety over an algebraically closed field of characteristic $p>0$. Then $\B_X$ is ample if and only if $\Omega_X$ is ample.
\end{thm}

In fact, we give versions of the two preceding theorems for nefness as well, via the exact same techniques.

\begin{rem}
In fact, everything in the paper works just as well if one replaces $F$ by $F^e$ for any $e\geq 1$, and considers the $e$-th Frobenius cokernel $\B_X^e := \coker(\O_X \to F_*^e \O_X)$. We have stuck to the case $e=1$ for simplicity, but the reader is welcome to substitute $p^e$, $F^e_*$, and $\B_X^e$ for $p$, $F_*$, and $\B_X$ throughout.
\end{rem}

\begin{rem}
As we will note in \cref{diffops}, positivity of $(\B_X^e)^\vee$ is closely related to the positivity of sheaves of level-$e$ differential operators in characteristic $p$. Thus, one way to view the results of \cite{CRP,CROZ} and this note is as a study of the positivity of these sheaves of differential operators.
\end{rem}

\subsection*{Acknowledgments}
The author was supported by EUR2023-143443 funded by MCIN/AEI/10.13039/\allowbreak 501100011033 and by FSE “invest in your future”, PID2021-125052NA-I00, funded by MCIN/AEI/10.13039/501100011033.
We thank Javier Carvajal-Rojas and Eamon Quinlan-Gallego for helpful comments on an earlier draft of this note.

\section{Preliminaries}
Here we give a brief overview of the necessary background material. 

\subsection{The Frobenius cokernel}
Let $X$ be a variety over an algebraically closed field $k$ of characteristic $p>0$. We write $F:X\to X$ for the (absolute) Frobenius on $X$. 
There is a canonical map of $\O_X$-modules
$$
\O_X \to F_* \O_X,
$$
which is injective since $X$ is reduced.

\begin{dfn}
The Frobenius cokernel, denoted $\B_X$, is
the cokernel of $\O_X\to F_* \O_X$.
\end{dfn}

\begin{rem}
Much of this note refers to the dual vector bundle $\B_X^\vee$.
Although we will not make use of this perspective here, $\B_X^\vee$ can be identified with the kernel of the Frobenius trace map $F_* (\omega_X^{1-p}) \to \O_X$. Thus, $\B_X^\vee$ is referred to as the Frobenius trace kernel. For more on this, see \cite[Section~2]{CRP}.
\end{rem}

\begin{rem}
The central question of \cite{CRP} is the ampleness of $\B_X^\vee$, or equivalently the antiampleness of $\B_X$.
We note the following facts about the positivity of $\B_X^\vee$ established in \cite{CRP}:
\begin{itemize}
\item (\cite[Proposition~5.9]{CRP}) If $X$ is a smooth projective variety  with $\B_X^\vee$ ample, then $X$ is Fano.
\item (\cite[Theorem~5.19]{CRP}) If $X$ is a threefold with $\B_X^\vee$ ample, then $\rho(X)=1$ .
\item (\cite[Proposition~4.2]{CRP}) If $f:X\to S$ is the blowup of a smooth variety $S$ along a smooth subvariety of codimension $>1$, then $\B_X^\vee$ is not ample.
\item (\cite[Proposition 5.14]{CRP}) If $X,S$ are smooth, $f:X\to S$ is a proper morphism with $f_*\O_X=\O_Y$, and the general fiber of $f$ is smooth, then if $\B_X^\vee$ is ample, then the general fiber of $f$ is 0-dimensional.
\end{itemize}
\end{rem}

\subsection{Cartier operators}
\label{Cartier}

Here we recall the notation and terminology of \cite[Section~2]{CRP}.
Let $X$ be a smooth variety over a perfect field $k$ of characteristic $p>0$. The de Rham complex $\Omega_X^\bullet$ is not $\O_X$-linear, but the Frobenius pushforward $F_* \Omega_X^\bullet$ is. 
There are  are canonical isomorphisms of $\O_X$-modules
$$
C\inv: \Omega_X^i \to \mathcal H^i(F_* \Omega_X^\bullet)
$$
(see \cite[Theorem~5.1]{Katz} for details), with inverse $C$, called the Cartier isomorphism.
If we write $\B_X^i :=  \im(F_*d: F_*\Omega_X^{i-1} \to F_*\Omega_X^{i})$ and $\mathcal Z_X^i:= \ker(F_* d: F_* \Omega_X^i \to F_* \Omega_X^{i+1})$
, then the isomorphism  $C$ gives rise to a short exact sequence of locally free $\O_X$-modules
$$
0 \to \B_X^i \to \mathcal Z_X^i \to \Omega_X^i \to 0.
$$

\begin{dfn}
The surjection $ \mathcal Z_X^i \to \Omega_X^i$ is called the Cartier operator $\k_i$.
\end{dfn}

We note for use in \cref{proofWTX} that there is a short exact sequence of locally free $\O_X$-modules 
$$
0\to \mathcal Z_X^i \to F_* \Omega_X^i \xra{F_* d} \B_X^{i+1} \to 0.
$$

\begin{rem}
Note that $\B_X^1$ is the Frobenius cokernel $\B_X$ defined above, since 
the inclusion
$ \mathcal Z_0 \hookrightarrow F_* \O_X $
is the inclusion of the $p$-th powers in $\O_X$, i.e., is just the Frobenius $\O_X\to F_*\O_X$,
 and thus
the
cokernel is $\B_X^1$
\end{rem}

\begin{rem}
In the discussion above, one can replace $F_*$ by $F_*^e$, and obtain the $e$-th Frobenius cokernel $\B_{X,e}^{i}$ and the $e$-th Cartier operator $\k_i^e: \mathcal Z_{X,e}^{i} \to \Omega_X^i$.
We will discuss only the case $e=1$, but our discussion below applies equally well to higher $e$.
\end{rem}

\subsection{Truncated symmetric powers}

We make no claims to originality anywhere in this section. For detailed references on the material here, see \cite{DotyWalker} in the representation-theoretic context, or \cite{Sun} in the algebraic-geometric context.
See also \cite[Section~2.2]{GaoRaicu} for similar calculations.

\begin{dfn}
Let $V$ be a vector space of dimension $c$. For any $l$, there is a subspace
$$
T^l(V) \subset V^{\otimes l}
$$
called the $l$-th truncated symmetric power of $V$, defined as follows:
$S_l$ acts on $V^{\otimes l}$ by permuting the factors.
For $k_1,\dots,k_m$, we set
$$
v(k_1,\dots,k_c)
= \sum_{\sigma \in S_l}(e_1^{\otimes k_1}\otimes\dots\otimes e_c^{\otimes k_c} ) \cdot \sigma \in V^{\otimes \sum k_i};
$$
 $T^l(V)$ is the subspace of $V^{\otimes l}$ generated by the elements of the form
$
v(k_1,\dots,k_c)
$
for $k_1+\dots+k_c = l$.
\end{dfn}

\begin{rem}
\label{remtrunc}
In other words, $T^l(V)$ is the image of the map
$$
v: V^{\otimes l} \to V^{\otimes l};
$$
note that $v$ clearly factors through $S^l(V)$, the $l$-th symmetric power of $V$.
Thus, for any $l$, there is a surjection $S^l(V) \to T^l(V)$.
One can check that 
the resulting map $S^l(V) \to T^l(V)$ 
sends
an element
$$
e_1^{a_1}\cdots e_c^{a_c} 
$$
to zero if 
$a_i\geq p$ for some $i$,
and thus for dimension reasons
the kernel of this surjection is generated by the elements of the form
$ e_1^{a_1}\cdots e_c^{a_c} $
with some $a_i\geq p$.
In other words, $T^l(V)$ is the cokernel of the canonical multiplication map
$$
F^* V \otimes
S^{l-p}(V) \to S^l(V),
$$
i.e., the quotient of $S^l(V)$ by the ideal generated by the image of $F^* V\otimes S^{l-p}(V)$.
\end{rem}

Note that $T^{c(p-1)+1}(V) = 0$, since any monomial of degree $c(p-1)+1$ must have some $a_i \geq p$.
We can identify the last nonzero truncated power with a more familiar object:

\begin{lem}
\label{vector}
Let $V$ be a vector space of dimension $c$. There is a canonical isomorphism
$$
T^{c(p-1)}(V) \cong (\det V)^{p-1}
$$
\end{lem}

\begin{proof}
Let $e_1,\dots,e_c$ be a basis of $V$. Then $T^{c(p-1)}(V)$ is generated by the monomials $e_1^{a_1}\cdots e_c^{a_c}$ where $0\leq a_i < p$ for all $i$ and $\sum a_i = c(p-1)$. The only such monomial is $e_1^{p-1} e_2^{p-1} \cdots e_c^{p-1}$, which generates a one-dimensional vector space. The isomorphism sends this monomial to $(e_1\wedge e_2 \wedge \cdots \wedge e_c)^{p-1}$.

To see that this is canonical, note that if we change the basis $e_i$ by an invertible matrix $A$, then both $e_1^{p-1} e_2^{p-1} \cdots e_c^{p-1}$ and $(e_1\wedge e_2 \wedge \cdots \wedge e_c)^{p-1}$
 change by $\det(A)^{p-1}$,
\end{proof}

Globalizing \Cref{vector} gives the following:

\begin{cor}
\label{global}
Let $X$ be a variety and $G$ a vector bundle of rank $c$ on $X$. Then there is a canonical isomorphism of vector bundles
$$
T^{c(p-1)}(G) \cong (\det G)^{p-1}.
$$
\end{cor}

\subsection{Duality and ampleness of Frobenius pushforwards}
\label{easylemmas}

Here, we prove some easy lemmas regarding duality and ampleness of Frobenius pushforwards.

\begin{lem}
\label{dual}
Let $Z$ be a Cohen--Macaulay projective variety over an $F$-finite field of characteristic $p>0$, and let $E$ be a vector bundle on $X$.
Then
$$
F_*(E)^\vee \cong F_*(E^\vee \otimes \omega_Z^{1-p}).
$$
\end{lem}

\begin{proof}
By Grothendieck duality  for the finite morphism $F$, we have
$ \SHom(F_* E, \omega_Z) \cong F_* \SHom(E,  \omega_Z)$,
Thus we have
$$
\displaylines{
F_*(E)^\vee = \SHom(F_* E, \O_Z) \cong \SHom(F_* E, \omega_Z) \otimes \omega_Z^{-1} 
\hfill\cr\hfill
\cong F_* \SHom(E,  \omega_Z) \otimes \omega_Z^{-1} 
\cong F_*(E^\vee \otimes \w_Z )\otimes \w_Z^{-1}\cong F_*(E^\vee \otimes \omega_Z^{1-p}).
\qedhere
}
$$
\end{proof}

\begin{lem}
\label{dualample}
If $E$ is an ample vector bundle on a Cohen--Macaulay projective variety $Z$ of dimension $>0$, then $\Hom(E,\O_Z)=0$.
\end{lem}

\begin{proof}
If there is a nonzero homomorphism $E \to \O_X$, then there is a nonzero surjection $\Sym^m E \to \O_X$ for all $m$: If the image of $E$ is a nonzero ideal sheaf $J\subset \O_Z$, then 
$\Sym^m E\to \Sym^m J $ is surjective by right-exactness of symmetric powers. There's always a multiplication surjection $\Sym^m J \to J^m$, and thus we obtain a nonzero homomorphism 
$$\Sym^m E \twoheadrightarrow \Sym^m J \twoheadrightarrow J^m \hookrightarrow \O_Z.$$
Thus, we have that
$$
H^0((\Sym^m E)^\vee) \cong \Hom(\Sym^m E,\O_Z) \neq 0.
$$
However, by Serre duality we have
$$
H^0((\Sym^m E)^\vee) = H^{\dim Z}(\Sym^m E \otimes \omega_Z)^\vee
$$
and thus
$H^{\dim Z}(\Sym^m E \otimes \omega_Z) \neq 0$ for all $m$.
But if $\dim Z>0$, this cohomology group vanishes for $m\gg 0$ by ampleness of $E$.
\end{proof}

The following lemma says essentially that if $F_* L$ is ample on a Fano variety, then $L$ is quite positive;  it is easy to see that if $F_*L$ is ample then so is $L$, but the following is a bit stronger. For example, it says that $F_* \O_{\P^n}(d)$ is ample only if $d>(n+1)(p-1)$.

\begin{lem}
\label{pushample}
Let $L$ be a line bundle on a smooth projective variety $Z$ over an $F$-finite field of characteristic $p>0$ such that $F_*L$ is ample.
Then $L^\vee \otimes \omega_Z^{1-p}$ is not effective.
\end{lem}

\begin{proof}
If $F_*L $ is an ample vector bundle, we must have $\Hom(F_*L,\O_Z)=0$ by \cref{dualample}.
Again, we use Grothendieck duality to write this as 
$$
H^0(F_*(L^\vee \otimes \omega_Z^{1-p})) = H^0(L^\vee \otimes \omega_Z^{1-p}).
$$
Thus, we must have that $H^0(L^\vee \otimes \omega_Z^{1-p})$, i.e., $L^\vee \otimes \omega_Z^{1-p}$ is not effective.
\end{proof}

\begin{rem}
Similarly, if $F_*L$ is nef and $A$ is any ample vector bundle, then $(F_* L) \otimes A = F_*(L\otimes A^p)$ is ample and so $H^0((L\otimes A^p)^\vee  \otimes \omega_Z^{1-p})=0$.
In particular, $F_* \O_{\P^n}(d)$ is nef only if $d> (n+1)(p-1)-p = n(p-1)-1$, or equivalently $d\geq n(p-1)$;  moreover, it is easy to compute that this example is sharp (at least when $p>n$).
\end{rem}

Finally, we combine these lemmas in the following special case:

\begin{cor}
\label{amplenoteff}
Let $Z$ be a smooth projective variety over an $F$-finite field of characteristic $p>0$. 
If $F_*(L)^ \vee$ is ample, then $L$ is not effective.
\end{cor}

\begin{proof}
If $(F_*L)^\vee$ is ample,
then by \cref{dual} $F_*(L)^\vee = F_*(L^\vee \otimes \omega_Z^{1-p})$ is ample, and so by \cref{pushample} we have that $(L^\vee \otimes \omega_Z^{1-p})^\vee \otimes \omega_Z^{1-p} = L$ is not effective.
\end{proof}

\section{An obstruction to ampleness of $\B_X^\vee$}
Here, we state and prove the main theorem of the introduction.

\begin{thm}
\label{main}
Let $X$ be a smooth projective variety over a field of characteristic $p>0$. If $X$ has a smooth positive-dimensional subvariety $Z$ such that $\det( N_{Z/X})^{1-p}$ is effective, then $\B_X^\vee$ is not ample.
\end{thm}

\begin{lem}
Let $Z\subset X$ be a subvariety of a variety $X$, defined by the ideal sheaf $I \subset \O_X$.
There is a short exact sequence of $\O_Z$-modules
$$
0 \to F_*(I/I^{[p]}) \to (F_*\O_X)\res Z \to F_* \O_Z \to 0,
$$
and this descends to a short exact sequence 
$$
0 \to F_*(I/I^{[p]}) \to \B_X\res Z \to \B_Z \to 0.
$$
\end{lem}

\begin{proof}
Take the obvious short exact sequence
$ 0 \to I/I^{[p]} \to \O_X/I^{[p]} \to \O_Z \to 0, $
and push this forward by $F_*$ to get
$ 0 \to F_*(I/I^{[p]}) \to F_* (\O_X/I^{[p]}) \to F_* \O_Z \to 0.  $
The middle term is just $F_*(\O_X) \otimes \O_X/I$, i.e.,  we have the desired short exact sequence
$$
0 \to F_*(I/I^{[p]}) \to (F_*\O_X)\res Z \to F_* \O_Z \to 0.
$$
Now, 
we have a commutative diagram
$$
\begin{tikzcd}
 &  & \O_Z \ar[r, equal] \ar[d, hook] & \O_Z \ar[d, hook] \\
0 \arrow{r} & F_*(I/I^{[p]}) \arrow{r} \arrow{d} & (F_*\O_X)\res Z \arrow{r} \ar[d,two heads] & F_* \O_Z \arrow{r} \ar[d,two heads] & 0 \\
0 \arrow{r} & K \ar[r] & \B_X\res Z \arrow{r} & \B_Z \arrow{r} & 0
\end{tikzcd}
$$
where $K$ is the kernel of the surjection $\B_X\res Z \to \B_Z$. A diagram chase shows that $F_*(I/I^{[p]})\to K$ is bijective, giving the second short exact sequence.
\end{proof}

\begin{cor}
\label{smooth}
If $Z$ is a smooth subvariety of a smooth variety $X$, then $F_*(I/I^{[p]})$ is a vector bundle on $Z$, and there is a surjection of vector bundles $\B_X\res Z^\vee \to F_*(I/I^{[p]})^\vee$.
In particular, if $\B_X^\vee$ is ample, so is $F_*(I/I^{[p]})^\vee$.
\end{cor}

\begin{proof}
Note that $F_* \O_Z$ and $F_* \O_X \res Z$ are vector bundles on $Z$, since $X$ and  $Z$ are smooth. Thus, $F_*(I/I^{[p]})$ is the kernel of a surjection of vector bundles, hence a vector bundle. Dualizing the short exact sequence gives the desired surjection (here, we use that $\B_Z$ is locally free so that the dual of an injection is a surjection). Finally, use that  $\B_X\res Z^\vee $ is the restriction of an ample vector bundle to $Z$, hence ample, and that quotients of ample vector bundles are ample.
\end{proof}

If we can find a smooth subvariety $Z$ such that $F_*(I/I^{[p]})^\vee$ is not ample, then $\B_X^\vee$ cannot be ample. To this end, we study the structure of $I/I^{[p]}$ as an $\O_Z$-module.

\begin{prop}
\label{filt}
Let $Z\subset X$ be a smooth codimension-$c$ subvariety of a smooth variety $X$, defined by the ideal sheaf $I \subset \O_X$.
The $\O_X$-module $G_Z := I/I^{[p]}$
has $I$-adic filtration
$$
I^{c(p-1)} G_Z \subset I^{c(p-1)-1} G_Z \subset  \cdots \subset I^2 G_Z \subset I G_Z \subset G_Z
$$
such that:
\begin{enumerate}
\item $I^{c(p-1)}G_Z = 0$.
\item $G_Z/IG_Z \cong I/I^2$, the conormal bundle of $Z$ in $X$.
\item \label{last} The induced map of $\O_Z$-modules $F_*(I^{c(p-1)-1} G_Z) \to F_*(I/I^{[p]}) $ has locally free cokernel.
\item The quotients $I^i G_Z / I^{i+1} G_Z $ are isomorphic to the $\O_Z$-modules $T^{i+1} (I/I^2)$, and in particular $I^{c(p-1)-1} G_Z \cong T^{c(p-1)}(I/I^2)\cong (\det(I/I^2))^{p-1}$ 
\end{enumerate}
\end{prop}

\begin{proof}
All claims are \'etale- or complete-local, as long as the identifications made are canonical, and so
we may assume that $X=\Spec k[x_1,\dots,x_n]$ and $I=(x_1,\dots,x_c)$.
Then $G_Z := I/I^{[p]}$ is a free $\O_Z$-module with basis the monomials $x_1^{a_1}\cdots x_c^{a_c}$ where $0\leq a_i < p$ for all $i$ and at least one $a_i$ is not zero. Note that the maximal degree of such a monomial is $c(p-1)$.

To see (1), note that
$I^{i-1} G_Z = I^i/I^{[p]}$ is generated by those monomials $x_1^{a_1}\cdots x_c^{a_c}$ with $\sum a_i \geq i$, and since $I^{c(p-1)+1}\subset I^{[p]}$, we have $I^{c(p-1)} G_Z = 0$.
For (2), 
$ G_Z / I G_Z \cong I/(I^{[p]}+I^2) = I/I^2$, since $I^{[p]} \subset I^2$.

For (3), we note $F_*(I/I^{[p]})$ is locally free by \Cref{smooth}; in our setting here, a basis is given by the monomials $x_1^{a_1}\cdots x_n^{a_n}$, where $0\leq a_i < p$ for all $i$ and at least one $a_i$ is not zero for $1\leq i \leq c$. The image of $F_*(I^{c(p-1)-1} G_Z)$ in $F_*(I/I^{[p]})$ is generated by the monomials $x_1^{p-1} x_2^{p-1} \cdots x_c^{p-1}\cdot x_{c+1}^{a_{c+1}}\cdots x_n^{a_n}$ with $0
\leq a_{c+i}< p$ for $i>0$; since these are just a subset of the basis elements of the free $\O_Z$-module $F_*(I/I^{[p]})$, the cokernel is also free.

For (4),
note that
$I^i G_Z / I^{i+1} G_Z$ is generated by those monomials $x_1^{a_1}\cdots x_c^{a_c}$ with $\sum a_i = i+1$ and $a_i < p$; 
this can be viewed in a basis-independent way as $\Sym^{i+1}(I/I^2)$ modulo the image under multiplication of 
$$F^*(I/I^2)\otimes \Sym^{i+1-p}(I/I^2),$$
and thus per the description in \Cref{remtrunc} is (canonically) isomorphic to the $(i+1)$-st truncated symmetric power of the conormal bundle $I/I^2$. 
\end{proof}

\begin{cor}
\label{surj}
There is a surjection of $\O_Z$-modules $F_*(I/I^{[p]})^\vee  \to F_*(\det(N_{Z/X})^{1-p})^\vee $, obtained by dualizing the inclusion of \eqref{last} of \cref{filt}.
\end{cor}

\begin{proof}
By \Cref{global}, $I^{c(p-1)} G_Z \cong (\det(I/I^2))^{p-1}$, so dualizing the inclusion of \eqref{last} of \cref{filt} (and rewriting $(I/I^2)^{p-1}$ as $N_{Z/X}^{1-p}$) gives the desired surjection.
\end{proof}

Putting this all together with the results of \cref{easylemmas} yields the main theorem:

\begin{proof}[Proof of \Cref{main}]
By \cref{smooth}, if $\B_X^\vee$ is ample so is $F_*(I/I^{[p]})^\vee$.
Thus, by \Cref{surj} so would be its quotient $F_*(\det(N_{Z/X}^\vee)^{p-1})^\vee$. By \cref{amplenoteff}, this would imply that $\det(N_{Z/X})^{1-p}$ is not effective, a contradiction.
\end{proof}

We record the following corollary, since it will be useful:

\begin{cor}
\label{curveobstruction}
If $X$ contains a subspace $Z$ isomorphic to $\P^{r}$  such that $-K_X^{n-r}.Z \leq r+1 $ then $\B_X^\vee$ is not ample.
In particular, if $X$ contains a smooth rational curve $C$ such that $-K_X.C \leq 2$ 
\end{cor}

\begin{proof}
The adjunction sequence $0\to T_Z\to T_X\res Z \to N_{C/X}\to 0$ implies that 
the line bundle $ \det(N_{C/X}) $
has degree $-K_X^{c}.Z -(r+1)$, and if this is $\leq 0$ then the line bundle $\det(N_{C/X})^{1-p}$ on $Z\cong \P^r$ has nonnegative degree and thus is effective. The result then follows from \Cref{main}.
In particular, if $Z\cong \P^1$ this says exactly that if $-K_X.C \leq 2$ then $\det(N_{C/X})^{1-p}$ is effective.
\end{proof}

\section{Fano threefolds}

\begin{thm}
\label{Fano3}
Let $X$ be a Fano threefold over an algebraically closed field of characteristic $p$. $\B_X^\vee$ is ample if and only if $X$ is $\P^3$ or the quadric threefold in characteristic $p\neq 2$.
\end{thm}

We note that the arguments of the following proof are certainly immediate to experts, and rely heavily on the characteristic-$p$ classification work done in \cite{liftvanishI,liftvanishII,FanoI,FanoII}.

\begin{proof}
The backwards direction (i.e., that a quadric in characteristic $p\neq 2$ has antiample Frobenius cokernel) is discussed in \cite[Corollary~4.8]{CRP}. It is elementary that $\P^3$ has antiample Frobenius cokernel.

Now, we prove the forwards direction.
By \cite[Theorem~5.19]{CRP}, we may assume that $\rho(X) = 1$.

First, we consider the case  where $-K_X$ is very ample and generates $\Pic(X)$. 
By \cite[Theorem~6.6]{liftvanishI}, if $-K_X$ is very ample, then $X$ contains a line, i.e., a smooth rational curve $C$ with $-K_X.C = 1$, and thus by \Cref{curveobstruction} $\B_X^\vee$ is not ample.

Now, assume $-K_X$ is not very ample but $-K_X$ generates $\Pic(X)$. By \cite[Theorem~7.3]{liftvanishII}, $X$ lifts to $W(k)$, and then by the same argument as \cite[Theorem~6.6]{liftvanishI} $X$ contains a line if its lift does. Since the lift is an index-1 Fano threefold with $-K_X$ not very ample, it is either:
\begin{itemize}
\item 
A Fano variety of type (1-1), which is a double cover of $\P^3$ with branch locus a divisor of degree 6, or equivalently a hypersurface of degree 6 in $\P(1,1,1,1,3)$.
\item A Fano variety of type (1-2), 
which is a double cover of a quadric in $\P^4$ with branch locus a divisor of degree 8.
\end{itemize}
In either case, the variety contains a line by \cite[Corollary~1.5]{Shokurovline}.

Consider the case where the index of $X$ is $>1$.  If $-K_X$ is very ample, then
 taking a smooth hyperplane section  yields a smooth del Pezzo surface $Y$. If $Y$ is not $\P^2$ or $\P^1\times \P^1$, then $Y$ contains a $(-1)$-curve $C$, by the usual classification of del Pezzo surfaces. In this case, the normal bundle sequence
$$
0\to N_{C/Y} \to N_{C/X} \to N_{Y/X}\res C \to 0
$$
implies that $\deg \det(N_{C/X}) = \deg\det N_{C/Y} + 1 = 0$, and thus again by \Cref{curveobstruction} $\B_X^\vee$ is not ample.
But if $Y$ is $\P^2$ or $\P^1\times \P^1$, then $X$ is either $\P^3$ or the quadric threefold (e.g., by index reasons).

If $-K_X$ is not very ample, then by \cite[Theorem~7.3]{FanoII}, $X$ still lifts to a del Pezzo threefold over $W(k)$, and then by the same argument as above $X$ contains a line if its lift does. Since the lift is a del Pezzo threefold with $-K_X$ not very ample, a hyperplane section must be a degree-1 del Pezzo surface, which contains a $(-1)$-curve by the classification of del Pezzo surfaces, and then the argument of the preceding paragraph implies the lift contains a line.
\end{proof}

\section{Complete intersections}

\begin{thm}
\label{CIs}
Let $X\subset \P^n$ be a complete intersection of hypersurfaces of degree $d_1,\dots,d_{c}\geq 2$ over a field of characteristic $p>0$. Assume $\dim X \geq 2$. If $\sum d_i = n-1$ or $n$ then $\B_X^\vee$ is not ample.
\end{thm}

\begin{rem}
Note that by \cite{CRP}, it is also true (but more immediate) that if $d\geq n+1$ then $\B_X^\vee$ is likewise not ample.
\end{rem}

\begin{proof}
Note that if $X$ contains a line $C$ of $\P^n$, then $-K_X.C= n+1 - \sum d_i$, and thus by \Cref{curveobstruction} 
if $\sum d_i \geq n-1$ then $-K_X.C\leq 2$ and thus $\B_X^\vee$ is not ample.
The fact that $X$ contains such a line is well-known: the Grassmannian of lines in $\P^n$ has dimension $2n-2$, and the condition of lying on a complete intersection of hypersurfaces of degrees $d_1,\dots,d_{c}$ is codimension $\sum (d_i+1)=\sum d_i +c$. 
Thus, $X$ will contain a line as long as 
$$
2n-2 \geq \sum d_i + c \geq n + c,
$$
or equivalently as long as $2\leq n-c = \dim X$.
\end{proof}

\begin{rem}
Note that by index reasons, if the index of $X$ is $\geq 3$, then $X$ can contain no $(-K_X)$-lines or conics, and the above argument breaks down. One can instead ask about higher-dimensional linear subvarieties, but a general complete intersection will not contain a linear subvariety of dimension large enough to preclude ampleness of $\B_X^\vee$.
However,
one can show this way (for example) that a cubic fourfold containing a plane cannot have ample $\B_X^\vee$.

We expect (from somewhat limited computational evidence) that no complete intersection other than projective space or a quadric has ample $\B_X^\vee$.
\end{rem}

\section{Cartier operators and antipositivity of~$\Omega^i_X$}

In \cite[Question 1.3]{CRP}, the authors raise the question of whether the antiampleness of the kernels of the $n+1$ Cartier operator $\k_i: F_* \mathcal Z_i \to \Omega_X^i$ can characterize projective space.
In particular, they raise the question of of the ampleness of $(\B^2_X)^\vee$ for quadric threefolds and $\P^3$. 
(For the definition of the Cartier operators, see \Cref{Cartier}.)
Here, we show that $(\B^2_X)^\vee$ (and hence all $(\B^i_X)^\vee$) is ample for both of these varieties, via the following:

\begin{prop}
\label{wedgeTX}
Let $X$ be a smooth variety such that $\bigwedge^i T_X$ is ample (respectively, nef). Then $(\B^i_X)^\vee$ is ample (respectively, nef).
\end{prop}

\begin{proof}
\label{proofWTX}
We have short exact sequences of vector bundles
$$
0 \to \B^i_X \to  \mathcal Z_X^i \xrightarrow{\k_i} \Omega_X^i \to 0
$$
and 
$$
0 \to \mathcal Z_X^i \to F_* \Omega_X^i \xrightarrow{F_*d^i} \B^{i+1}_X \to 0.
$$
Dualizing, we obtain surjections 
$(\mathcal Z^i_X)^\vee \twoheadrightarrow (\B^i_X)^\vee$
and 
$(F_*\Omega_X^i)^\vee \twoheadrightarrow (\mathcal Z^i_X)^\vee$,
and thus a surjection
$$
(F_* \Omega_X^i)^\vee \twoheadrightarrow (\B^i_X)^\vee.
$$
It thus suffices to show that $(F_* \Omega_X^i)^\vee$ is ample. This follows from the assumption that $\bigwedge^i T_X = (\Omega_X^i)^\vee$ is ample: \cite[Theorem~1.1]{antiample} states that if $E^\vee$ is ample and $f:X\to Y$ a finite surjection, then $(f_* E)^\vee$ is ample as well; though the proof there is written with a characteristic-$0$ assumption, it carries over verbatim in characteristic~$p$.
\end{proof}

The following gives a large class of examples where $\bigwedge^i T_X$ is ample for $i$ large compared to the index $a$ of a Fano variety $X$ in $\P^n$:

\begin{prop}
Let $X\subset \P^n$ be a smooth subvariety such that $\w_X = \O_X(-a)$ for $a>0$.
Then $\bigwedge^i T_X$ is ample for all $\dim X-a+2\leq i \leq \dim X$, and nef for all $\dim X - a +1 \leq i \leq \dim X$.
\end{prop}

For  $X$ a degree-$d$ hypersurface in $\P^n$, we have $a=n+1-d $ and thus:

\begin{cor}
\label{TXhyper}
Let $X$ be a smooth degree-$d$ hypersurface in $\P^n$. Then $\bigwedge^i T_X$ is ample for all $d\leq i \leq n-1$ and nef for all $d-1 \leq i \leq n-1$.
In particular, if $Q$ is a smooth quadric hypersurface, then $\bigwedge^i T_Q$ is ample for all $2\leq i \leq n-1$.
\end{cor}

The proof of the proposition is a straightforward application of Castelnuovo--Mumford regularity and the cotangent bundle sequence:

\begin{proof}
Since $\bigwedge^i T_X = (\bigwedge^i \Omega_X )^\vee$
the perfect pairing 
$
\bigwedge^i \Omega_X \otimes \bigwedge^{\dim X-i} \Omega_X \to \w_X
$
implies that $$\bigwedge^i T_X \cong \bigwedge^{\dim X - i} \Omega_X \otimes \w_X\inv
=\Bigl(\bigwedge^{\dim X - i} \Omega_X\Bigr) (a).
$$

Now, we claim that $\bigwedge^k \Omega_X (j)$ is ample for $j\geq k +2$ and nef for $j\geq k+1$.
To see this, note that there is a surjection
$$
\bigwedge^k \Omega_{\P^n}(j)\res X  \to \bigwedge^k \Omega_X (j) \to 0,
$$
Thus, it suffices to prove that $\bigwedge^k \Omega_{\P^n}(j)$ is ample for $j \geq k+2$ and nef for $j\geq k+1$.
This would complete the proof of the proposition, since taking $k = \dim X -i$ we get that $\bigwedge^i T_X = \bigwedge^{\dim X- i}\Omega_X(a)$ is ample for $i \geq \dim X-a+2$ and nef for all $i\geq \dim X - a +1$.

In fact, we will show that
$\bigwedge^k \Omega_{\P^n}(k+1)$ is globally generated, thus nef; then
 twisting the surjection
$$
\O_{\P^n}^{\oplus N} \to \bigwedge^k \Omega_{\P^n}(k+1) \to 0
$$
by $\O_{\P^n(1)}$ we have that $\bigwedge^k \Omega_{\P^n}(k+2)$ is a quotient of an ample vector bundle and hence ample.

So, we need to show that $\bigwedge^k \Omega_{\P^n}(k+1)$ is globally generated; by the theory of Castelnuovo-Mumford regularity, it suffices to show that $\bigwedge^k \Omega_{\P^n}$ is $k+1$-regular, i.e., that
$H^i(\bigwedge^k \Omega_{\P^n}(k+1-i)) = 0$ for all $i>0$.
However, by the Bott formula \cite[Proposition~14.4]{Bott}
we have
for $i<n$ that
$H^i( \bigwedge^k \Omega_{\P^n}(j)) $ is nonzero only for $j=0$ and $i=k$, 
while for $i=n$ it is nonzero only for $j<k-n$.
Thus, it is clear that $H^i(\bigwedge^k \Omega_{\P^n}(k+1-i)) = 0$ for all $i>0$, and the result follows.
\end{proof}

Combining this with the well-known fact that $T_{\P^n}$, and hence all its exterior powers, are ample, we have the following:

\begin{cor}
If $X$ is $\P^n$ or a smooth quadric hypersurface, then $(\B^i_X)^\vee$ is ample for all $i\geq 2$.
\end{cor}

Note also that $(\B^1_X)^\vee$ is ample if $X$ is a smooth quadric hypersurface of dimension $\geq 3$ and $p>2$, by \cite[Corollary~4.8]{CRP}; thus, the antiampleness of all kernels of the Cartier operators cannot distinguish between $\P^n$ and a smooth quadric.

\section{Ampleness of $\B_X^\vee$}

In this section we consider instead the ampleness of the Frobenius cokernel. We show that it is ample if and only if the cotangent bundle is ample.

\begin{rem}
\label{amplepush}
Note that $F_* \O_X$ is never ample if $\dim X > 0$: If it were, then so would be the pullback $F^* F_* \O_X$. However, the canonical map $F^* F_* \O_X\to \O_X$ can easily be verified to be a surjection of vector bundles, and hence $\O_X$ would be ample.
\end{rem}

\begin{thm}
\label{ample}
Let $X$ be a smooth projective variety over an $F$-finite field of characteristic $p>0$. Then $\B_X$ is ample (respectively, nef) if and only if $\Omega_X$ is ample (respectively, nef).
\end{thm}

\begin{proof}
Since $F$ is finite, ampleness of $\B_X$ is equivalent to ampleness of $F^* \B_X$, and likewise for nefness; since pullback is right exact, $F^* \B_X=  \coker(\O_X \to F^* F_* \O_X )$.
Since $X$ is smooth,
by \cite{Katz,Sun}, there is a filtration on $F^* F_* \O_X$
$$
0= V_{n(p-1)+1} \subset V_{n(p-1)} \subset \cdots \subset V_1 \subset V_0 = F^* F_* \O_X
$$ 
such that
\begin{itemize}
\item $V_1 = \ker(F^* F_* \O_X \to \O_X)$. 
\item 
 $V_i/V_{i+1} \cong  T^i \Omega_X^1$ for all $i$.
\end{itemize}
By \cref{amplepush}, $F^* F_* \O_X = \O_X\oplus F^* \B_X$, with $F^* \B_X$ exactly the kernel of $F^* F_* \O_X\to \O_X$. Thus, $F^* \B_X = V_1$ has a filtration with graded pieces $T^i \Omega_X^1$ for $1\leq i \leq n(p-1)$.

Now, if  $\B_X$, and hence $F^* \B_X$, are ample (or nef), then the first short exact sequence of the filtration is 
$$
0 \to V_2 \to F^* \B_X \to T^1 \Omega_X =\Omega_X \to 0,
$$
and thus $\Omega_X$ is ample (or nef).

If $\Omega_X$ is ample (or nef), then we claim all the $T^i \Omega_X$ are ample (or nef). Once shown, the ampleness or nefness of $F^* \B_X$ is immediate: The last short exact sequence of the filtration is
$$
0 \to V_{n(p-1)} = T^{n(p-1)} \Omega_X^1 \to V_{n(p-1)-1}  \to T^{n(p-1)-1} \Omega_X \to 0, 
$$
and thus $V_{n(p-1)}$ is an extension of ample (or nef) bundles and hence ample (or nef) by \cite[Corollary~3.4]{HartshorneVB}. Continuing inductively, we see that all the $V_i$ are ample (or nef), and in particular $V_1=F^* \B_X$ is ample (or nef), and thus $\B_X$ is ample (or nef).
\end{proof}

\begin{rem}
The condition of $\Omega_X$ being ample is quite strong. For example, it is inherited by all subvarieties, and thus such an $X$ can contain no rational or elliptic curves.
There are several known classes of varieties with ample cotangent bundle, including:
\begin{enumerate}
\item Curves of genus $g\geq 2$.
\item Complete intersections of at least $n/2$ sufficiently ample general hypersurfaces in an abelian variety of dimension $n$ \cite{Debarre},  or likewise in $\P^n$ \cite{BD,Xie,CIRE}.
\item A high-codimension general linear section of a product of varieties with big cotangent bundle (\cite{Debarre}, originally due to Bogomolov).
\end{enumerate}
With that said, to our knowledge there is no prospective classification for such varieties, and thus in characteristic $p$ no classification of varieties with ample Frobenius cokernel.
\end{rem}

\section{Connection to Frobenius differential operators}
\label{diffops}

In the final section, we make a comment on the relation between positivity of $\B_X^\vee$ and positivity of sheaves of differential operators in characteristic $p$.

We briefly recall the notion of differential operators:

\begin{dfn}
Let $k$ be a field,
let $R$ be a $k$-algebra. We define $R$-modules $D^m_{R/k}\subset \End_k(R)$, the $k$-linear differential operators of order $m$ inductively as follows:
\begin{itemize}
\item $D^0_{R/k} =\Hom_R(R,R)\cong R$, thought of as multiplication by $R$. 
\item  $\delta \in \End_k(R)$ is in $D^m_{R/k}$ if $[\delta, r] \in D^{m-1}_{R/k}$ for any $r\in D^0_{R/k}$.
\end{itemize}
We write $D_{R/k}=\bigcup D^m_{R/k}$. $D_{R/k}$ is a noncommutative ring, and $R$ is a left $D_{R/k}$-module.
\end{dfn}

There is another natural filtration of $D_R$ in characteristic $p$:

\begin{prop}[{{\cite[1.4.9]{Yekutieli}}}]
\label{Yek}
Let $k$ be a perfect field of characteristic $p>0$, and 
and  let $R$ be an $F$-finite $k$-algebra. Then 
$$
D_{R/k} = \bigcup_{e} \Hom_{R^{p^e}}(R,R).
$$
We write $D_R^{(e)} := \Hom_{R^{p^e}}(R,R)$, and refer to them as differential operators of level $e$. 
\end{prop}

All these notions globalize to any variety $X$, and so we write $D^m_X$ and $D^{(e)}_X$ for the sheaves of order-$m$ and level-$e$ differential operators on $X$, respectively.

Note that $D^m_X$ and $D^{(e)}_X$ always have a free summand $\O_X$
(the natural subsheaf $D^{(0)}_X = D^0_X = \O_X $ splits via the ``evaluation at $1$'' map $D^m_X\to \O_X$ or $D^{(e)}_X \to \O_X$).
, and so are never ample if $\dim X>0$.
Thus, it is natural to ask when the quotients (equivalently, complementary summands) $D^m_X/\O_X$ or $D^{(e)}_X/\O_X$ is ample.

Here, we make the following observations:

\begin{prop}
$D^m_X/\O_X$ is ample if and only if $X=\P^n$.
\end{prop}

\begin{proof}
By induction on $m$. For $m=1$, $D^1_X = \O_X\oplus T_X$, and by \cite{Mori} the ampleness of $T_X$ is equivalent to $X=\P^n$.
For $m>1$, there is a short exact sequence
$$
0 \to D^{m-1}_X \to D^m_X \to (\Sym^m \Omega_X)^\vee \to 0,
$$
and the inclusion is an isomomorphism of the subsheaf $\O_X$ of $D^{m-1}_X$ to the subsheaf $\O_X$ of $D^m_X$. Thus, there is an induced short exact sequence
$$
0 \to D^{m-1}_X/\O_X \to D^m_X/\O_X \to \Sym^m T_X \to 0.
$$
Thus, if $D^m_X/\O_X$ is ample, then so is $(\Sym^m \Omega_X)^\vee$.

If we could write $(\Sym^m \Omega_X)^\vee = \Sym^m T_X$, then 
ampleness of
$\Sym^m T_X$ would imply that $T_X$ is ample by \cite[Proposition~2.4]{HartshorneVB}, and so again \cite{Mori} would force $X=\P^n$. In characteristic 0, this is true.

However, in characteristic $p$ it is not true that 
$(\Sym^m \Omega_X)^\vee = \Sym^m T_X$ (the left side is the so-called ``$m$-th divided power of $T_X$''). However, as noted on \cite[p.~594]{Mori}, all that his proof requires is that:
\begin{enumerate}
\item $-K_X$ is ample.
\item For any $f:\P^1\to X$, $f^* T_X$ is a direct sum of ample line bundles.
\end{enumerate}
These both follow from ampleness of $(\Sym^m \Omega_X)^\vee$:
\begin{enumerate}
\item 
The determinant of an ample bundle is ample, and 
since taking determinants commutes with taking duals, the determinant of $(\Sym^m \Omega_X)^\vee$ is the determinant of $\Sym^m T_X$, which is a positive multiple of $-K_X$, and thus $-K_X$ is ample.
\item 
Let $f:\P^1\to X$ be any morphism. Then $f^* (\Sym^m \Omega_X)^\vee$ is ample, and thus the direct sum of positive-degree line bundles. Thus $f^* \Sym^m \Omega_X$ is a direct sum of negative-degree line bundles, and so $f^* \Omega_X$ must be a direct sum of negative-degree line bundles. Thus, $f^* T_X$ is a direct sum of positive-degree line bundles.
\end{enumerate}
Thus, Mori's proof goes through, and $X=\P^n$.
\end{proof}

By contrast, ampleness of $D^{(e)}_X/\O_X$ is a less stringent condition: 

\begin{prop}
$D^{(e)}_X/\O_X$ is ample if and only if $(\B_X^e)^\vee$ is ample.
\end{prop}

\begin{proof}
On the one hand
$F^* F_* \O_X = \O_X\oplus F^* \B_X$, so that
$$
\SHom(F^* F_* \O_X,\O_X) = \O_X\oplus F^* \B_X^\vee.
$$

On the other hand,
$$
\SHom(F^* F_* \O_X,\O_X) 
= \SHom_{\O_X}(\O_X \otimes_{\O_{X^{p^e}}} \O_X, \O_X)
= \SHom_{\O_{X^{p^e}}}(\O_X , \O_X) = D^{(e)}_X
$$
is nothing but the sheaf of level-$e$ differential operators, with the trivial summand $\O_X$ of $\SHom(F^* F_* \O_X,\O_X)$ mapping isomorphically to the trivial summand $\O_X$ of $D^{(e)}_X$. 

Thus, there is an induced isomorphism
$$
F^* \B_X^\vee \cong D^{(e)}_X/\O_X,
$$

 Since $\B_X^\vee$ is ample if and only if $F^* \B_X^\vee $ is ample, the claim follows.
\end{proof}

\begin{rem}
Thus, we have that $D^m_X/\O_X$ ample for any $m$ implies that $X=\P^n$ and thus that $D_X^{(e)}/\O_X$ is ample for all $e$. The converse fails any time $(B_X^e)^\vee$ is ample but $X\neq \P^n$; for example, for quadrics in dimension $\geq 3$ and characteristic $\neq 2$.
Thus, one can view this fact as an ``unexpected positivity'' of the sheaf of level-$e$ differential operators, when compared against the order-$m$ differential operators.
\end{rem}

\bibliographystyle{alpha}
\bibliography{link}

\end{document}